\documentclass[11pt,oneside]{article}

\usepackage{graphicx} 
\usepackage[centertags]{amsmath}
\usepackage{amssymb}
\usepackage{amsthm}
\usepackage[english]{babel}

\newtheorem{thm}{Theorem}[section]
\newtheorem{cor}[thm]{Corollary}

\newtheorem{prop}[thm]{Proposition}
\theoremstyle{definition}
\newtheorem{defn}[thm]{Definition}

\numberwithin{equation}{section}
\newcommand{\Real}{\mathbb R}

\newcommand{\To}{\longrightarrow}
\newcommand{\A}{\mathcal{A}}
\newcommand{\punt}{\mathbf{.}}

\newcommand{\ibs}{\boldsymbol{i}}
\newcommand{\varbs}{\boldsymbol{\varepsilon}}
\newcommand{\jbs}{\boldsymbol{j}}
\newcommand{\lbs}{\boldsymbol{l}}

\newcommand{\kbs}{\boldsymbol{k}}
\newcommand{\zbs}{\boldsymbol{z}}
\newcommand{\mbs}{\boldsymbol{m}}
\newcommand{\xbs}{\boldsymbol{x}}

\newcommand{\vbs}{\boldsymbol{v}}

\newcommand{\deltabs}{\boldsymbol{\delta}}

\newcommand{\X}{\boldsymbol{X}}

\newcommand{\mubs}{\boldsymbol{\mu}}
\newcommand{\nubs}{\boldsymbol{\nu}}
\newcommand{\gabs}{\boldsymbol{\gamma}}
\newcommand{\debs}{\boldsymbol{\delta}}
\newcommand{\lambs}{\boldsymbol{\lambda}}
\newcommand{\chibs}{\boldsymbol{\chi}}
\newcommand{\bio}{\boldsymbol{\iota}}
\newcommand{\boeta}{\boldsymbol{\eta}}
\newcommand{\trasp}{\scriptscriptstyle{T}}

\begin{document}
\title{Multivariate time-space harmonic polynomials: a symbolic approach.}
\author{E. Di Nardo \footnote{Department of Mathematics, Computer Sciences and Economics,
University of Basilicata, Viale dell'Ateneo Lucano 10, 85100 Potenza, Italy, elvira.dinardo@unibas.it}, I.
Oliva \footnote {Department of Economics, University of Verona, Via dell'Artigliere 19, 37129 Verona, Italy, 
immacolata.oliva@univr.it}}
\date{\today}
\maketitle
\begin{abstract}
By means of a symbolic method, in this paper we introduce a new family of multivariate polynomials such that 
multivariate L\'evy processes can be dealt with as they were martingales. In the univariate case, this family of polynomials is known as time-space harmonic polynomials. Then, simple closed-form expressions of some multivariate classical families of polynomials are given. The main advantage of this symbolic representation is the plainness of  the setting which reduces to few fundamental statements but also of its implementation  in any symbolic software. The role played by cumulants is emphasized within the generalized Hermite
polynomials. The new class of multivariate L\'evy-Sheffer systems is introduced.
\end{abstract}
\textsf{\textbf{keywords}: umbral calculus, multivariate L\'evy process, multivariate time-space harmonic polynomial, cumulant}
%------------------------------------------------------------------------------------
\section{Introduction}
%------------------------------------------------------------------------------------
In mathematical finance, a multivariate stochastic process $\X_t$ in ${\mathbb R}^d$ usually models the price process of hedging portfolios at time $t.$ 
The tools employed to work with $\X_t$ are either its probability distribution 
or its characteristic function. When the probability distribution of $\X_t$ is known analytically, numerical 
algorithms are resorted aiming to compute the contingent claim's price $E[\varphi(\X_t)]$ with $\varphi$ some payoff function. 
When the characteristic function of $\X_t$ is known analytically, fast Fourier transform or Monte Carlo methods are applied to 
multivariate integrals 
$$E[\varphi(\X_t)] = \int_{{\mathbb R}^d} \hat{\varphi}(\boldsymbol{u}) \, E[\exp(i \boldsymbol{u} \X_t)] \, d \boldsymbol{u},$$
with $\hat{\varphi}$ the Fourier transform of $\varphi.$ Both methodologies require implementations which can take some time
due to the involved multivariate integrals.  

Recently \cite{cuchiero}, a class of processes, called \emph{polynomial processes}, has been introduced
as follows: assume $S$ the state space of $\X_t,$ a 
closed subset of ${\mathbb R}^d.$ Then there exists a multivariate polynomial $Q(\xbs,t)$ in 
\begin{equation}
{\hbox{Pol}}_{\leq m} (S) = \left\{ \left. \sum_{|\kbs|=0}^m c_{\kbs} \xbs^{\kbs} \right| \xbs \in S, c_{\kbs} \in {\mathbb R}\right\}
\label{spaziopol}
\end{equation}
such that a martingale property holds \footnote{In (\ref{spaziopol}), by using the multi-index notation $\xbs^{\kbs} = x_1^{k_1} x_2^{k_2} \cdots x_d^{k_d}$ and $c_{\kbs} = c_{k_1, k_2, \ldots, k_d}.$  Recall that a ($d$-dimensional) multi-index $\kbs$ is a $d$-tuple $\kbs = (k_1, \ldots, k_d)$ of nonnegative integers, such 
that $|\kbs| = k_1 + \cdots + k_d,$ and $\kbs! = k_1! \cdots k_d!.$}
\begin{equation}
E[Q(\X_t,t) | \X_s] = Q(\X_s,s)
\label{martingala}
\end{equation}
for $s \leq t.$ The martingales $\{Q(\X_t,t)\}$ are called polynomial processes.  The polynomials (\ref{spaziopol}) originated in conjunction with random matrix theory and multivariate statistics of ensembles. 
These processes are employed together with the reduction--variance method for the pricing and the hedging of some bounded measurable European claims. In \cite{cuchiero}, the attention is essentially focused on the properties shared by the class of stochastic processes $Q(\X_t,t),$ as for example affine processes or Feller processes with quadratic squared diffusion coefficients.  The computation of their coefficients requires the computation of a matrix exponential\footnote{If $X$ is a real or complex $n\times n$ matrix, the {\it matrix exponential} of $X$ is a $n \times n$ matrix $e^X$ whose power series is $e^X = \sum_{k \geq 0} X^k/k!.$}. In order to characterize the polynomials $Q$, Haar measure and zonal polynomials are involved. 
A computational efficient way to deal with zonal polynomials is not yet available: in particular, as coefficients of hypergeometric functions, zonal polynomials have manageable expressions only on the unitary group.  In the univariate case, the polynomials (\ref{spaziopol}) have been deeply analyzed by different authors, see \cite{solè} and references therein. They are called \emph{time-space harmonic} polynomials. For
L\'evy processes $\X_t,$ the main advantage of employing the polynomial process $Q(\X_t,t)$ is the martingale property 
(\ref{martingala}), fundamental in the martingale pricing \cite{PDE}, which not necessarily holds for L\'evy processes. 

In order to characterize multivariate time-space harmonic polynomials, we propose the multivariate Fa\`a di Bruno formula as main tool. In the univariate case, the Fa\`a di Bruno gives the $m$-th derivative of a composite function,
that is, if $f$ and $g$ are functions with a sufficient number of derivatives, then
$$\frac{g^{(m)}[f(t)]}{m!} = \sum \frac{g^{(k)}[f(t)]}{b_1! b_2! \dots b_m!} 
\left(\frac{f^{(1)}(t)}{1!}\right)^{b_1} \left( \frac{f^{(2)}(t)}{2!} \right)^{b_2} \dots \left( \frac{f^{(m)}(t)}{m!} \right)^{b_m},$$
where the sum is over all $m$-tuples of nonnegative integers $(b_1, \ldots, b_m)$ such that 
$b_1 + 2b_2 + \cdots + m b_m = m$ and $k = b_1 + \cdots + b_m.$ 
This formula is extended to the multivariate case in \cite{dibruno} by using multi-index partitions.
Its implementation is quite cumbersome, but recently an optimized algorithm has been 
introduced \cite{dibruno} by using a symbolic method, known in the literature as the classical umbral calculus \cite{SIAM}. 
Its main device is to represent number sequences by suitable symbols via a linear functional, resembling the expectation of random 
variables (r.v.'s). Thanks to this symbolic method, the multivariate Fa\`a di Bruno formula is computed by using 
suitable multivariate polynomials whose indeterminates are replaced by different polynomials. In this paper we characterize a bases for the space of multivariate time-space harmonic polynomials (\ref{spaziopol})
involving L\'evy processes. These polynomials have a simple expression, easily implementable in any 
symbolic software by using the algorithms addressed in \cite{dibruno}.
   
Special families of multivariate polynomials such as the Hermite polynomials, the Bernoulli polynomials and the Euler polynomials are then recovered. The new class of multivariate L\'evy-Sheffer systems is introduced. The remainder of the paper is organized in order
to resume terminology, notations and some basic definitions of the symbolic method. For the symbolic univariate L\'evy process introduced in \cite{TSH}, we add two more examples: the stable and the inverse Gaussian processes. 

Since orthogonal polynomials are currently employed in mathematical finance \cite{schoutens}, an interesting application which 
deserves further deepening studies is the orthogonal property of multivariate time-space harmonic polynomials.
Some preliminarily results are given in \cite{dinardotsh} and \cite{lawi}. These are certainly connected to multivariate Sheffer sequences 
\cite{brown} whose treatment would indeed benefit of an umbral approach. Finally, since the symbolic method has already been 
applied within random matrix theory \cite{DinardoMc} in studying their cumulants, we believe fruitful to employ 
this setting early to matrix-valued polynomial processes. 
%------------------------------------------------------------------------------------
\section{The symbolic method}
%------------------------------------------------------------------------------------

The symbolic method we refer is a syntax consisting of a set $\A = \{\alpha, \beta, \gamma, \ldots\}$ 
of symbols called \emph{umbrae} and a linear functional called \emph{evaluation} $E\,:\, \Real[x][\A]\To \Real[x],$ with $\Real$ the set of real numbers, such that
$E[1] = 1$ and 
$$E[x^n \, \alpha^i \, \beta^j \, \cdots \, \gamma^k ]=x^n \, E[\alpha^i] \, E[\beta^j] \, \cdots \,  E[\gamma^k]
\qquad \hbox{(\emph{uncorrelation property})}$$ 
for all distinct umbrae $\alpha, \beta, \ldots, \gamma \in \A$ and for all nonnegative integers $n, i, j, \ldots k.$  A  unital sequence of real numbers $a_0 = 1, a_1, a_2, \ldots$ is said to be \emph{umbrally represented} by an umbra $\alpha$ if $E[\alpha^k] = a_k, \, \mbox{ for all } k \geq 0.$ By analogy with moments of a r.v., the $k$-th element 
of the sequence $\{a_k\}$ is called the $k$-th moment of the umbra $\alpha.$

Two distinct umbrae $\alpha$ and $\gamma$ are said to be \emph{similar} if and only if they represent the same sequence of moments
$$\alpha \equiv \gamma \,\Leftrightarrow \, E[\alpha^k] = E[\gamma^k], \quad \mbox{ for } k = 0, 1, 2, \ldots.$$

A polynomial $p \in \Real[\A]$ is called an \emph{umbral polynomial}. The \emph{support} of $p$ is the set of all 
occurring umbrae of $\A.$ Two umbral polynomials $p$ and $q$ are said to be \emph{umbrally equivalent} if and only if $E[p] = E[q],$ in symbols $p \simeq q.$ They are uncorrelated when their supports are disjoint. 

The formal power series
\begin{equation}
u + \sum_{k \geq 1} \alpha^k \frac{z^k}{k!} \in \Real[\A][z]
\label{gf1}
\end{equation}
is the \emph{generating function} of the umbra $\alpha$ and it is denoted by $e^{\alpha z}.$ By extending coefficientwise the notion of umbral equivalence, any exponential formal power series
\begin{equation}
f(z) = 1 + \sum_{k \geq 1} a_k \frac{z^k}{k!}
\label{gf2}
\end{equation}
can be umbrally represented by a formal power series (\ref{gf1}). Indeed, if the sequence $\{a_k\}$ is 
umbrally represented by an umbra $\alpha,$ we have $f(z) = E[e^{\alpha z}].$  The formal power series (\ref{gf2}) is denoted by $f(\alpha, z)$
to underline the role played by the moments of the umbra $\alpha.$ Special umbrae are 
\begin{description}
\item[{\it a)}] the \emph{singleton umbra} $\chi$ with generating function $f(\chi, z) = 1 + z;$ 
\item[{\it b)}] the \emph{unity umbra} $u$ with generating function $f(u,z) = e^z;$
\item[{\it c)}] the \emph{augmentation umbra} $\varepsilon$ with generating function $f(\varepsilon,z) = 1;$
\item[{\it d)}] the \emph{Bell umbra} $\beta$ with generating function $f(\beta, z) = \exp\{e^z - 1\}.$ 
\end{description}

An \emph{auxiliary umbra} is a symbol not in the alphabet $\A$ but defined in such a way that it represents special sequences of moments. Then, the alphabet $\A$ is extended by inserting these auxiliary symbols as they were elements 
of $\A.$ For example, by the auxiliary symbol $\alpha \dot{+} \gamma,$ we denote 
the \emph{disjoint sum} of two umbrae, representing the sequence of coefficients of 
$f(\alpha,z) + f(\gamma,z) - 1.$ Therefore we write $f(\alpha \dot{+} \gamma,z) = f(\alpha,z) + f(\gamma,z) - 1$
and deal with $\alpha \dot{+} \gamma$ as it was an element of the alphabet $\A.$ 
One more example, which will be largely employed in the rest of the paper, is
the \emph{dot-product} $n \punt \alpha$ of a nonnegative integer $n$ and an umbra $\alpha.$ Let us  
consider the set $\{\alpha^{\prime}, \alpha^{\prime\prime}, \ldots, \alpha^{\prime\prime\prime}\}$ 
of $n$ distinct umbrae, similar to the umbra $\alpha$ but uncorrelated each other. Define 
the symbol $n \punt \alpha$ as $n \punt \alpha = \alpha^{\prime} + \alpha^{\prime\prime} + \cdots + \alpha^{\prime\prime\prime}.$
Then, we have $f(n \punt \alpha, z) = [f(\alpha, z)]^n.$ The moments of $n \punt \alpha$ are (see \cite{Dinsen} for further details) 
\begin{equation}
E[(n \punt \alpha)^k] = \sum_{i = 1}^k (n)_i B_{k,i}(a_1, a_2, \ldots, a_{k - i + 1}), \qquad \hbox{for all $k \geq 1,$}
\label{dot_mom1}
\end{equation}
where $B_{k,i}(a_1, a_2, \ldots, a_{k - i + 1})$ are the partial exponential Bell polynomials \cite{comtet}, $(n)_i 
= n (n - 1) \cdots (n - i + 1)$ is the lower factorial and $\{a_j\}$ are moments of $\alpha.$ 

Equation (\ref{dot_mom1}) suggests a way to define new auxiliary umbrae depending on a real parameter $t.$ Indeed, 
the $k$-th moment of the dot-product $n \punt \alpha$ is a polynomial, say $q_k(n),$ of degree $k$ in $n.$ 
The integer $n$ can be replaced by any $t \in \Real$ so that
$$
q_k(t) = \sum_{i = 1}^k (t)_i B_{k,i}(a_1, a_2, \ldots, a_{k - i + 1}) \quad \hbox{for all $k \geq 1,$}
$$
still denotes a polynomial of degree $k$ in $t.$ The symbol $t \punt \alpha$ is then introduced as the 
auxiliary umbra such that $E[(t \punt \alpha)^k] = q_k(t),$ for all nonnegative integers $k.$ This 
symbol is called the dot-product of $t$ and $\alpha.$ Its generating function is 
$f(t \punt \alpha, z) = [f(\alpha, z)]^t.$ By using similar arguments, the real parameter $t$ could be replaced by any umbra $\gamma.$ 
The umbra $\gamma \punt \alpha$ representing $\{E[q_k(\gamma)]\}_{k \geq 0}$  
is the dot-product of $\gamma$ and $\alpha$ and $f(\gamma \punt \alpha, z) = f(\gamma, \log[f(\alpha, z)]).$
These replacements are the main device of the symbolic method. For example, we can replace the umbra 
$\gamma$ by the auxiliary umbra $t \punt \beta$ such that $E[(t \punt \beta)_k] = t^k$ for all $k \geq 1,$ see \cite{Bernoulli}. Then we have 
\begin{equation}
q_k(t \punt \beta \punt \alpha) = \sum_{i = 1}^k t^i \, B_{k,i}(a_1, a_2, \ldots, a_{k - i + 1}) \quad \hbox{for all $k \geq 1,$}
\label{dot_mom_umbr1}
\end{equation}
and $(t \punt \beta) \punt \alpha \equiv t \punt (\beta \punt \alpha).$ We omit the parentheses when 
they are not necessary for the computations. The generating function of $t \punt \beta \punt \alpha$ is
\begin{equation}
f(t \punt \beta \punt \alpha,z) = \exp\{t[f(\alpha,z) - 1]\}.
\label{(genfun1)}
\end{equation}
Again, we can replace $t$ in (\ref{dot_mom_umbr1}) with an umbra $\gamma.$ The umbra $\gamma \punt \beta
\punt \alpha$ is the composition umbra of $\gamma$ and $\alpha,$ see \cite{Bernoulli}. 
%-----------------------------------------------------------------------------------------------------------
\subsection{Symbolic L\'evy processes}
%-----------------------------------------------------------------------------------------------------------
Recall that a stochastic process $\{X_t\}_{t \geq 0}$ on $\Real$ is a \emph{L\'evy process} if its increments $X_{t} - X_{t-1}$ are independent and stationary r.v.'s. 
L\'evy processes share the infinite divisibility property \cite{sato} and their symbolic representation generalizes this own property. 
Indeed the following theorem has been proved in \cite{TSH}.  
\begin{thm} \label{T1}
Let $\{X_t\}_{t \geq 0}$ be a L\'evy process with finite moments for all $t$ and let $\alpha$ be an umbra such that $f(\alpha,z) = E[e^{z X_1}].$ Then for any $t \geq 0,$ the moment sequence of the L\'evy process $\{X_t\}_{t \geq 0}$ is umbrally represented by the family of auxiliary umbrae $\{t \punt \alpha\}_{t \geq 0}.$ 
\end{thm}
Theorem \ref{T1} states that if the moments of an umbra $\alpha$ are all finite, the family of auxiliary umbrae $\{t \punt \alpha\}_{t \geq 0}$ is the umbral counterpart of a L\'evy process $\{X_t\}_{t \geq 0}$ such that $E[X_t^k] = E[(t \punt \alpha)^k],$ for all nonnegative integers $k.$ 
In \cite{noncentered} and \cite{TSH}, several examples of symbolic L\'evy processes have been given. Here we add two more
examples.  
%------------------------------------------------------------------------------------
\paragraph{The $m$-stable process.}
%------------------------------------------------------------------------------------
A $m$-stable process is a L\'evy process whose increments are independent, stationary and $m$-stable distributed. 
In particular the $m$-stable process is an example of stochastic process with not convergent moment generating function
in any neighborhood of zero. Nevertheless, since the symbolic method asides from the convergence of 
formal power series (\ref{gf2}), we are able to characterize L\'evy processes which are $m$-stables. 
\begin{defn} \label{def2}
An umbra $\alpha \in \A$ is said to be \emph{stable} with \emph{stability parameter} $m \in [0, 2)$ if there exist $b_n,\, c_n 
\in \Real$ such that $n \punt \alpha \equiv c_n \alpha + b_n \punt u,$ where $u$ is the unity umbra.
\end{defn}

Previous definition parallels the same given in probability theory \cite{levy1, levy2} for stable r.v.'s  
In \ref{def2}, we have $c_n = n ^{1/m}$ by using similar arguments given in \cite{levy2}. From Definition 
\ref{def2}, we have $$b_n \punt u \equiv n \punt \alpha + n ^{1/m} (- 1 \punt \alpha)$$
where $-1 \punt \alpha$ is the inverse of the umbra $\alpha,$ that is the auxiliary umbra such that
$ \alpha + (- 1 \punt \alpha) \equiv \varepsilon.$

\begin{defn} The family of auxiliary umbrae $\{t \punt \alpha\}_{t \geq 0},$ where $\alpha$ is a stable umbra,
is a symbolic stable L\'evy process. 
\end{defn}

%-----------------------------------------------------------------------------------------------------------------------
\paragraph{Inverse Gaussian process.}
%-----------------------------------------------------------------------------------------------------------------------
An inverse Gaussian process $\{X_t^{\scriptscriptstyle{(IG)}}\}_{t \geq 0}$ is a L\'evy process with independent, stationary and 
inverse Gaussian distributed increments \cite{chhikara}, that is $X_t^{\scriptscriptstyle{(IG)}} \sim IG(a,b)$
with $a, b > 0.$ The moment generating function of $X_t^{\scriptscriptstyle{(IG)}}$ is 
\begin{equation}
E\left[e^{z X_t^{\scriptscriptstyle{(IG)}}}\right] = \exp\left\{t \frac{b}{a}\left[ 1 - \left( \sqrt{1 - \frac{2 a^2 z}{b}} \right) 
\right]\right\}.
\label{gf_IG}
\end{equation}

%The name \emph{inverse Gaussian} depends on the fact that, as Tweedie observed in \cite{tweedie}, there exists an inverse relationship 
%between the cumulant generating functions of this distribution of probability and the cumulant generating function of gaussian random 
%variables. More precisely, if $X$ is an inverse gaussian random variable and $Y$ is a Gaussian random variable, we have
%\begin{equation}
%\log(E[e^{-zX}]) = (\log(E[e^{-zY}]))^{\scriptscriptstyle{<-1>}}.
%\label{gfc_prop}
%\end{equation}
%
%We use this property to introduce the \emph{inverse Gaussian umbra}.
To obtain the symbolic expression of an inverse Gaussian process, we need to recall the notion of compositional inverse of an umbra. Indeed, if $\alpha$ is an umbra
with generating function $f(\alpha,z),$ then the compositional inverse of $\alpha$ is the umbra 
$\alpha^{\scriptscriptstyle{<-1>}}$ such that 
$\alpha \punt \beta \punt \alpha^{\scriptscriptstyle{<-1>}} \equiv \alpha^{\scriptscriptstyle{<-1>}} \punt \beta \punt \alpha \equiv \chi.$ 
In particular its generating function is such that $f(\alpha^{\scriptscriptstyle{<-1>}},z) = f^{\scriptscriptstyle{<-1>}}(\alpha,z)$ where
\begin{equation}
f[\alpha, f^{\scriptscriptstyle{<-1>}}(\alpha,z) - 1] =  f^{\scriptscriptstyle{<-1>}}[\alpha,f(\alpha,z) -1] = 1 + z.
\label{(compinv)}
\end{equation}

\begin{defn} \label{defIG}
An umbra $\alpha$ is said to be an \emph{inverse Gaussian umbra} if
$-\alpha \equiv \beta \punt (-b \chi \dot{+} \sqrt{a} \delta)^{\scriptscriptstyle{<-1>}}$
with $\chi$ the singleton umbra and $\delta$ the Gaussian umbra, such that
$f(\delta,z)= 1 + z^2/2.$
\end{defn}

Set $\sqrt{a}=s.$ Definition \ref{defIG} moves from the generating function given in (\ref{gf_IG}). Indeed  we have
\begin{align*}
f(-\alpha, z) = f[\beta \punt (-b \chi \dot{+} s \delta)^{\scriptscriptstyle{<-1>}},z] = \exp\{f[(-b \chi \dot{+} s \delta)^{\scriptscriptstyle{<-1>}},z] - 1\}.
\end{align*}

\noindent
For the sake of simplicity, set $f^{\scriptscriptstyle{<-1>}}(-b \chi \dot{+} s \delta,z) = \bar{f}(z).$ From (\ref{(compinv)}) we have
$f(-b \chi \dot{+} s \delta, \bar{f}(z) - 1) = 1 + z.$ Since  
$f(-b \chi \dot{+} s \delta, z) = (1 - bz) + s^2 z^2 /2$ 
then $s^2 {\bar{f}}^2(z) - 2 (b + s^2) \bar{f}(z) + 2b +  s^2 - 2 z = 0.$
Solving the previous equation with respect to $\bar{f}(z)$ and 
replacing $s^2=a$ we obtain 
\begin{equation} 
\bar{f}(z) = 1 + \frac{b}{a}\left[1 \pm \sqrt{1 +  \frac{2 z a}{b^2}}\right].
\label{sol}
\end{equation}
As $\bar{f}(0)=1,$ see (\ref{gf2}), in (\ref{sol}) we choose the minus sign so that
$$f[\beta \punt (-b \chi \dot{+} s \delta)^{\scriptscriptstyle{<-1>}},z] = \exp\left\{\frac{b}{a}\left[1 - \sqrt{1 +  \frac{2 z a}{b^2}}\right]\right\}.$$
By using (\ref{(genfun1)}), we recover the generating function (\ref{gf_IG}).

\begin{defn}
The family of auxiliary umbrae $\{t \punt \beta \punt (-b \chi \dot{+} \sqrt{a} \delta)^{\scriptscriptstyle{<-1>}}\}_{t \geq 0}$ is the symbolic inverse Gaussian process.  
\end{defn}
%
%----------------------------------------------------------------------------------------------------------------
\section{Multivariate time-space harmonic polynomials}
%----------------------------------------------------------------------------------------------------------------
In the univariate classical umbral calculus, the main device is to replace $a_n$ with $\alpha^n$ via the linear evaluation $E.$ In the same way, the main device of the multivariate symbolic method is to replace sequences like $\{g_{i_1,i_2,...,i_d}\}$ with a product of powers 
$\mu_1^{i_1} \mu_2^{i_2} \ldots \mu_d^{i_d},$ where $\{\mu_1, \mu_2, \ldots, \mu_d\}$ are umbral monomials and $i_1, \ldots, i_d$ 
are nonnegative integers. Note that the supports of the umbral monomials in $\{\mu_1, \mu_2, \ldots, \mu_d\}$ are not necessarily disjoint. A multivariate version of this symbolic method has been given first in \cite{dibruno}. 

A sequence $\{g_{\vbs}\}_{\vbs \in \mathbf{N}_0^d}$ with $g_{\vbs} = g_{v_1, \ldots, v_d}$ and $g_{\bf 0} = 1$ is umbrally represented by the $d$-tuple $\mubs$ if $E[\mubs^{\vbs}] = g_{\vbs},$
for all $\vbs \in \mathbb{N}_0^d,$ with $\mubs^{\vbs} =  \mu_1^{v_1} \mu_2^{v_2} \ldots \mu_d^{v_d}.$ Then $g_{\vbs}$ is called the \emph{multivariate moment} of $\mubs.$ 

Two $d$-tuples $\mubs$ and $\nubs$ of umbral monomials are said to be similar if they represent the same sequence of multivariate moments, 
in symbols $E[\mubs^{\vbs}] = E[\nubs^{\vbs}], \mbox{ for all } \vbs \in \mathbb{N}_0^d.$ Two $d$-tuples $\mubs$ and $\nubs$ of umbral monomials are said to be uncorrelated if  
$E[\mubs^{\vbs_1} \nubs^{\vbs_2}] = E[\mubs^{\vbs_1}] E[\nubs^{\vbs_2}] , \mbox{ for all } \vbs_1, \vbs_2 \in \mathbb{N}_0^d.$

The exponential multivariate formal power series
\begin{equation}
e^{\mubs \zbs^{\trasp}} = \boldsymbol{u} + \sum_{k \geq 1} \sum_{\substack{\vbs \in \mathbb{N}_0^d \\ |\vbs| = k}} 
\mubs^{\vbs} \frac{\zbs^{\vbs}}{\vbs!}
\label{multi_gf}
\end{equation}
is said to be the \emph{generating function} of the $d$-tuple $\mubs.$ Now, assume $\{g_{\vbs}\}_{\vbs \in \mathbb{N}_0^d}$ umbrally 
represented by the $d$-tuple $\mubs.$ If the sequence $\{g_{\vbs}\}_{\vbs \in \mathbb{N}_0^d}$ has exponential multivariate generating 
function
$$ f(\mubs, \zbs) = 1 + \sum_{k \geq 1} \sum_{\substack{\vbs \in \mathbb{N}_0^d \\ |\vbs| = k}} g_{\vbs} \frac{\zbs^{\vbs}}{\vbs!},$$
suitably extending coefficientwise the action of $E$ to the generating function (\ref{multi_gf}), we have
$E[e^{\mubs \zbs^{\trasp}}] = f(\mubs, \zbs).$ Henceforth, when no confusion occurs, we refer to $f(\mubs, \zbs)$ as the generating 
function of the $d$-tuple $\mubs.$

As done in the univariate case, we can introduce the auxiliary umbra $n \punt \mubs.$ To express its moments,
the notion of multi-index partition \cite{dibruno} needs to be recalled. 

\begin{defn} \label{partition_multi}
A partition $\lambs$ of a multi-index $\vbs,$ in symbols $\lambs \vdash \vbs,$ is a 
matrix $\lambs = (\lambda_{ij})$ of nonnegative integers and with no zero columns in lexicographic order
$\prec$ such that $\lambda_{r_1} + \lambda_{r_2} + \cdots + \lambda_{r_k} = v_r$ for $r = 1, 2, \ldots , d.$ 
\end{defn}
The number of columns of $\lambs$ is denoted by $l(\lambs)$. The notation $\lambs
= (\lambs_{1}^{r_1} , \lambs_{2}^{r_2}, \ldots)$ represents the
matrix $\lambs$ with $r_1$ columns equal to $\lambs_{1},$ $r_2$ columns equal to $\lambs_{2}$ and so on, 
where $\lambs_{1} \prec  \lambs_{2} \prec \ldots.$ We set $\mathfrak{m}(\lambs) = (r_1, r_2, \ldots),$ $\mathfrak{m}(\lambs)! = r_1! r_2! \cdots$ and $\lambs! = \lambs_1! \lambs_2! \cdots.$
Then, for a nonnegative integer $n$ we have
\begin{equation}
(n \punt \mubs)^{\vbs} \simeq \sum_{\lambs \vdash \vbs} \frac{\vbs!}{\mathfrak{m}(\lambs) \lambs!} \, (n)_{l(\lambs)} \, 
\mubs_{\lambs},
\label{(auxmult)}
\end{equation}
where $\mubs_{\scriptscriptstyle{\lambs}} = (\mubs_{\scriptscriptstyle{\lambs}}^{\prime\lambs_1})^{\punt r_1} 
(\mubs_{\scriptscriptstyle{\lambs}}^{\prime\prime\lambs_2})^{\punt r_2} \ldots,$ with $\mubs^{\prime}, \mubs^{\prime\prime}, 
\dots$ uncorrelated $d$-tuple similar to $\mubs.$
By replacing $n$ with the real parameter $t,$ we have 
\begin{equation}
(t \punt \mubs)^{\vbs} \simeq \sum_{\lambs \vdash \vbs} \frac{\vbs!}{\mathfrak{m}(\lambs) \lambs!} \, (t)_{l(\lambs)} \, 
\mubs_{\lambs},
\label{(auxmult1)}
\end{equation}
while by replacing $n$ with the auxiliary umbra $t \punt \beta$ we obtain
\begin{equation}
(t \punt \beta \punt \mubs)^{\vbs} \simeq \sum_{\lambs \vdash \vbs} \frac{\vbs!}{\mathfrak{m}(\lambs) \lambs!} 
t^{l(\lambs)} \, \mubs_{\lambs}.
\label{composition_multi}
\end{equation}
The auxiliary umbrae $t \punt \mubs$ and $t \punt \beta \punt \mubs$ are the building blocks in dealing with symbolic L\'evy processes and multivariate time-space harmonic polynomials. For their definition, the conditional evaluation with respect to an umbral $d$-tuple $\mubs$ needs to be introduced. 
\begin{defn}
Assume ${\mathcal{X}} = \{\mu_1, \mu_2, \ldots, \mu_d\}.$ The linear operator 
$$E(\;\cdot \; \vline \,\, \mubs): \, \Real[x_1, \ldots, x_d][\A] \; \longrightarrow \; \Real[\mathcal{X}]$$ 
such that $E(1 \,\, \vline \,\, \mubs) = 1$ and 
$$E(x_1^{l_1} \, x_2^{l_2} ,\ \cdots \, x_d^{l_d} \, \mubs^{\ibs} \, \nubs^{\jbs} \, \gabs^{\kbs} \cdots \, \,  \vline \,\, \mubs) 
= x_1^{l_1} \, x_2^{l_2} \, \cdots \, x_d^{l_d} \, \mubs^{\ibs} \, E[\nubs^{\jbs}] \, E[\gabs^{\kbs}] \cdots$$ 
for uncorrelated $d$-tuples $\mubs, \nubs, \gabs \ldots,$ for $\mbs, \ibs, \jbs,  \ldots \in \mathbb{N}_0^d$ and $\{l_i\}_{i=1}^d$ nonnegative integers, is called \emph{conditional evaluation} with respect to the umbral $d$-tuple $\mubs.$
\end{defn}

\begin{defn}
Let $\{P(\xbs,t)\} \in \Real[x_1, \ldots, x_d]$ be a family of polynomials indexed by $t \geq 0.$ The polynomial $P(\xbs,t)$ is said to be a 
\emph{multivariate time-space harmonic polynomial} with respect to the family of auxiliary umbrae $\{t \punt \mubs\}_{t \geq 0}$ 
if and only if
\begin{equation}
E \left( P(t \punt \mubs, t) \, \, \vline \, \, s \punt \mubs \right) = P(s \punt \mubs,s), \;\; \mbox{ for all } s \leq t.
\label{def_tsh_multi}
\end{equation}
\end{defn}
Since $f[(n + m) \punt \mubs, \zbs] = f(\mubs, \zbs)^{n + m} = f(n \punt \mubs, \zbs) \, f(m \punt \mubs, \zbs),$ then 
\begin{equation}
(n + m) \punt \mubs \equiv n \punt \mubs + m \punt \mubs^{\prime},
\label{(aaa1)}
\end{equation}
with $\mubs$ and $\mubs^{\prime}$ uncorrelated $d$-tuples of umbral monomials. Then, for 
$E(\;\cdot \; \vline \,\, \mubs)$ is reasonable to assume  
\begin{equation}
E[\{(n + m) \punt \mubs\}^{\vbs} \,\, \vline \,\, n \punt \mubs] = E[\{n \punt \mubs + m \punt \mubs^{\prime}\}^{\vbs} \,\, 
\vline \,\, n \punt \mubs],
\label{(aaa)}
\end{equation}
for all nonnegative integers $n, m$ and for all $\vbs \in \mathbb{N}_0^d.$ If $n \ne m,$ then 
equation (\ref{(aaa)}) gives
\begin{equation}
E[\{(n + m) \punt \mubs\}^{\vbs} \,\, \vline \,\, n \punt \mubs] = E[\{n \punt \mubs + m \punt \mubs\}^{\vbs} \,\, 
\vline \,\, n \punt \mubs],
\label{(aaa1)}
\end{equation}
since $n \punt \mubs$ and $m \punt \mubs$ are uncorrelated auxiliary umbrae. We will use (\ref{(aaa1)})
when no misunderstanding occurs.  Thanks to equation (\ref{(aaa)})
we have
\begin{align}
\nonumber
E\left[\{(n + m) \punt \mubs\}^{\vbs} \,\, \vline \,\, n \punt \mubs\right] & = E\left[\sum_{\kbs \leq \vbs} \binom{\vbs}{\kbs} 
(n \punt \mubs)^{\kbs} (m \punt \mubs)^{\vbs - \kbs} \,\, \vline \,\, n \punt \mubs \right] \\
& = \sum_{\kbs \leq \vbs} \binom{\vbs}{\kbs} (n \punt \mubs)^{\kbs} E[(m \punt \mubs)^{\vbs - \kbs}],
\label{ce_dot_multi}
\end{align}
where $\kbs \leq \vbs \; \Leftrightarrow \; k_i \leq v_i, \mbox{ for all } i = 1, \ldots, d$ and 
$\binom{\kbs}{\vbs} = \binom{k_1}{v_1} \cdots \binom{k_d}{v_d}.$
By analogy with (\ref{(aaa1)}) and (\ref{ce_dot_multi}), we have $(t + s) \punt \mubs \equiv t 
\punt \mubs + s \punt \mubs$ and for $t \geq 0$
\begin{equation}
E\left[(t \punt \mubs)^{\vbs} \,\, \vline \,\, s \punt \mubs \right] = \sum_{\kbs \leq \vbs} \binom{\vbs}{\kbs} (s \punt \mubs)^{\kbs} 
E[\{(t - s) \punt \mubs\}^{\vbs - \kbs}].
\label{condeval_dot_multi}
\end{equation}
As $s,t \in {\mathbb R},$ recall that the auxiliary umbra $-t \punt \mubs$ denotes the inverse of $t \punt \mubs$ that is $- t \punt \mubs + t \punt \mubs \equiv \varbs$ where $\varbs$ is the $d$-tuple such that $\varbs=(\varepsilon_1, \varepsilon_2, \ldots, \varepsilon_d),$ with $\{\varepsilon_i\}$ uncorrelated augmentation umbrae.

Theorem 3.6 allows us to introduce the class of multivariate time-space harmonic polynomials with respect to a symbolic $d$-dimensional L\'evy process. Symbolic $d$-dimensional L\'evy processes have been introduced in \cite{multilevy}. Here we recall
the main results.

\begin{defn} \label{def_levy_multi}
A stochastic process $\{\X_t\}_{t \geq 0}$ on $\Real^d$ is a \emph{multidimensional L\'evy process} if
\begin{description}
	\item[{\rm (i)}] $\X_0 = \boldsymbol{0}$ a.s.
	
	\item[{\rm (ii)}] For all $n \geq 1$ and for all $0 \leq t_1 \leq t_2 \leq \ldots \leq t_n < \infty,$ the r.v.'s $\X_{t_2} - \X_{t_1}, \X_{t_3} - \X_{t_2}, \ldots$ are independent.
	
	\item[{\rm (iii)}] For all $s \leq t,$ $\X_{t + s} - \X_s \stackrel{d}{=} \X_t.$
	
	\item[{\rm (iv)}] For all $\varepsilon > 0,$ $\lim_{h \rightarrow 0} P(|\X_{t + h} - \X_t| > \varepsilon) = 0.$
	
	\item[{\rm (v)}] $t \mapsto \X_t (\omega)$ are c\'adl\'ag, for all $\omega \in \varOmega,$ with $\varOmega$ the underlying sample space.   
\end{description}
\end{defn}

The moment generating function of a \emph{multidimensional L\'evy process}
is $\varphi_{\scriptscriptstyle{t}}(\zbs) = [\varphi_{\scriptscriptstyle{\X_1}}(\zbs)]^t$ with  
$\varphi_{\scriptscriptstyle{\X_1}}(\zbs) = E\left[e^{\zbs \X_1^{\trasp}}\right]$ and
$\zbs \in \Real^d.$ If we denote by $\mubs$ the $d$-tuple such that $f(\mubs,\zbs)=\varphi_{\scriptscriptstyle{\X_1}}(\zbs),$ then symbolic multidimensional L\'evy processes can be constructed as done in the previous section. 

\begin{thm} \label{primo}
A L\'evy process $\{\X_t\}_{t \geq 0}$ in $\Real^d$ is umbrally represented by $\{t \punt \mubs\}_{t \geq 0},$ where 
$\mubs$ is such that $g_{\vbs} = E[\mubs^{\vbs}] = E\left[\X_1^{\vbs} \right],$ for all $\vbs \in \mathbb{N}_0^d.$
\end{thm}

\begin{thm}\label{UTSH2_multi}
For all $\vbs \in \mathbb{N}_0^d,$ the family of polynomials
\begin{equation}
Q_{\vbs}(\xbs,t) = E[(\xbs - t \punt \mubs)^{\vbs}] \in \Real[x_1, \ldots, x_d]
\label{(tshumbral_multi)}
\end{equation}
is time-space harmonic with respect to $\{t \punt \mubs\}_{t \geq 0}.$
\end{thm}
\begin{proof}
We need to prove that the family of polynomials $Q_{\vbs}(\xbs,t)$ satisfies equality (\ref{def_tsh_multi}). Observe that
$$Q_{\vbs}(\xbs,t) = E\left[ \sum_{\kbs \leq \vbs} \binom{\vbs}{\kbs} \xbs^{\vbs - \kbs} (-t \punt \mubs)^{\kbs} \right] 
= \sum_{\kbs \leq \vbs} \binom{\vbs}{\kbs} \xbs^{\vbs - \kbs} E[(-t \punt \mubs)^{\kbs}],$$
where $\xbs^{\lbs} = x_1^{l_1} x_2^{l_2} \cdots x_d^{l_d}$ for $\lbs = (l_1, l_2, \ldots, l_d) \in \mathbb{N}_0^d.$
Thus,
\allowdisplaybreaks
\begin{equation}
Q_{\vbs}(t \punt \mubs,t) = \sum_{\kbs \leq \vbs} \binom{\vbs}{\kbs} (t \punt \mubs)^{\vbs - \kbs} E[(-t \punt \mubs)^{\kbs}].
\label{tsh_mul}
\end{equation}

By applying (\ref{tsh_mul}), we have
\begin{align*}
 E\left[Q_{\vbs}(t \punt \mubs,t) \; \vline \,\, s \punt \mubs\right] & =  \sum_{\kbs \leq \vbs} \binom{\vbs}{\kbs} E\left[(t 
\punt \mubs)^{\vbs - \kbs} E[(-t \punt \mubs)^{\kbs}] \; \vline \,\, s \punt \mubs\right] \\
& = \sum_{\kbs \leq \vbs} \binom{\vbs}{\kbs} E\left[(t \punt \mubs)^{\vbs - \kbs} \; \vline \,\, s \punt \mubs\right] E[(-t \punt 
\mubs)^{\kbs}]. 
\end{align*}
The result follows by suitably rearranging the terms and using (\ref{condeval_dot_multi}) 
\begin{align*}
& E\left[Q_{\vbs}(t \punt \mubs,t) \; \vline \,\, s \punt \mubs\right]  = \\
& = \sum_{\kbs \leq \vbs} \binom{\vbs}{\kbs} \left\{ \sum_{\jbs \leq \vbs - \kbs} \binom{\vbs - \kbs}{\jbs} (s \punt \mubs)^{\jbs} 
E[\{(t - s) \punt \mubs^{\prime}\}^{\vbs - \kbs - \jbs}] \right\} E[(-t \punt \mubs)^{\kbs}] \\
& = \sum_{\kbs \leq \vbs} \binom{\vbs}{\kbs} (s \punt \mubs)^{\kbs} E[\{(t - s) \punt \mubs^{\prime} + (-t \punt \mubs)\}^{\vbs - \kbs}] \\
& = \sum_{\kbs \leq \vbs} \binom{\vbs}{\kbs} (s \punt \mubs)^{\kbs} E[(-s \punt \mubs)^{\vbs - \kbs}] = Q_{\vbs}(s \punt \mubs, s).
\end{align*}
\end{proof}
A feature of a multivariate time-space harmonic polynomial is that when $\xbs$ is replaced by $t \punt \mubs$ its
overall evaluation is zero.  
\begin{cor} \label{zero} $E[Q_{\vbs}(t \punt \mubs,t)] = 0.$
\end{cor}
Assume 
\begin{equation}
Q_{\vbs}(\xbs,t) = \sum_{\kbs \leq \vbs} q_{\scriptscriptstyle{\kbs}}(t) \, \xbs^{\kbs}.
\label{coeff1_multi}
\end{equation}
Some properties of the coefficients $q_{\scriptscriptstyle{\kbs}}(t)$ are stated in the following propositions.
\begin{cor}\label{cor1_multi}
$q_{\scriptscriptstyle{\kbs}}(t) = \binom{\vbs}{\kbs}E[(-t \punt \mubs)^{\vbs - \kbs}]$ for all $\kbs \leq \vbs$ and 
$q_{\scriptscriptstyle{\kbs}}(0) = 0$ for all $\kbs < \vbs.$
\end{cor}

\begin{proof} The first equality follows by comparing (\ref{coeff1_multi}) with (\ref{tsh_mul}). For the second one,
if $t = 0$ then $ 0 \punt \mubs \equiv \boldsymbol{\varepsilon},$ and in particular
$E[(0 \punt \mubs)^{\vbs - \kbs}] = E[\boldsymbol{\varepsilon}^{\vbs-\kbs}] = 1$ if $\kbs = \vbs$ and $0$ otherwise.
\end{proof}
\begin{prop} \label{prop1_multi}
$q_{\scriptscriptstyle\kbs}(t - 1) = \sum_{\kbs \leq \jbs \leq \vbs} \binom{\jbs}{\kbs} \, g_{\jbs} \, q_{\scriptscriptstyle\jbs}(t),$
for all $\, \kbs < \vbs.$
\end{prop}

\begin{proof}
The result follows from Corollary \ref{cor1_multi} as 
\begin{align*}
q_{\scriptscriptstyle\kbs}(t - 1) & =  \binom{\vbs}{\kbs}E[(-t 
\punt \mubs + \mubs)^{\vbs-\kbs}] = \binom{\vbs}{\kbs} \sum_{\jbs \leq \vbs - \kbs} \binom{\vbs-\kbs}{\jbs} E\left[ (-t \punt \mubs)^{\vbs-\kbs-\jbs}\right] 
g_{\jbs} \\
& = \sum_{\kbs \leq \ibs \leq \vbs} \binom{\ibs}{\kbs} g_{\ibs - \kbs} E\left[\binom{\vbs}{\ibs} (-t \punt \mubs)^{\vbs-\ibs}\right].
\end{align*}
\end{proof}

\begin{cor} $g_{\vbs} \, q_{\scriptscriptstyle\kbs}(t) = q_{\scriptscriptstyle{\boldsymbol 0}}(t-1) - \sum_{\jbs < \vbs} g_{\jbs} \, 
q_{\scriptscriptstyle\kbs}(t). $
\end{cor}

\begin{proof}
The result follows from Proposition \ref{prop1_multi}, as
\begin{align*}
q_{\scriptscriptstyle{\boldsymbol 0}}(t-1) & = \sum_{\jbs \leq \vbs} \binom{\jbs}{\boldsymbol 0} q_{\scriptscriptstyle\jbs}(t) \,
g_{\jbs} = \sum_{\jbs \leq \vbs} q_{\scriptscriptstyle\jbs}(t) \, g_{\jbs} = q_{\scriptscriptstyle\vbs}(t) g_{\vbs} + \sum_{\jbs < \vbs} q_{\scriptscriptstyle\jbs}(t) \, g_{\jbs}.
\end{align*}
\end{proof}

\begin{prop} 
$g_{\vbs} = q_{\boldsymbol 0}(t - 1) - \sum_{\kbs < \vbs} q_{\scriptscriptstyle\kbs}(t) \, g_{\kbs}.$
\end{prop}

\begin{proof}
From Corollary \ref{cor1_multi}, we have
$$\sum_{\kbs < \vbs}  q_{\scriptscriptstyle\kbs}(t) \, g_{\kbs} = \sum_{\kbs \leq \vbs} \binom{\vbs}{\kbs} E[(-t \punt \mubs)^{\vbs-\kbs}] \, g_{\kbs} - g_{\vbs} 
= E[(-(t-1) \punt \mubs)^{\vbs}] - g_{\vbs}.$$
The result follows by applying again Corollary \ref{cor1_multi} to $E[(-(t-1) \punt \mubs)^{\vbs}].$
\end{proof}
The following theorem characterizes the coefficients of any multivariate time-space harmonic polynomial. 
\begin{thm} \label{comb_multi}
A polynomial 
\begin{equation}
P(\xbs,t) =\sum_{\kbs \leq \vbs} p_{\scriptscriptstyle\kbs}(t) \, \xbs^{\kbs}
\label{polynomial}
\end{equation}
is a time-space harmonic polynomial with respect to $\{t \punt \mubs\}_{t \geq 0}$ if and only if
\begin{equation}
p_{\scriptscriptstyle\kbs}(t) = \sum_{\kbs \leq \ibs \leq \vbs}  \binom{\ibs}{\kbs} \, p_{\scriptscriptstyle\kbs}(0) \, 
E[(-t \punt \mubs)^{\ibs-\kbs}], \quad \hbox{ for } \kbs \leq \vbs. 
\label{prop5_multi}
\end{equation}
\end{thm}

\begin{proof}
Suppose $P(\xbs,t)$ as in \eqref{polynomial} with coefficients \eqref{prop5_multi}. Suitably rearranging the 
terms in the summation (\ref{prop5_multi}), we have
\begin{align}
P(\xbs,t) = \sum_{\kbs \leq \vbs} p_{\scriptscriptstyle\kbs}(0) \sum_{\ibs \leq \kbs} \binom{\kbs}{\ibs} E[(-t \punt \mubs)^{\kbs-\ibs}] 
\xbs^{\ibs} = \sum_{\kbs \leq \vbs} p_{\scriptscriptstyle\kbs}(0) Q_{\kbs}(\xbs, t) \label{(comb)}.
\end{align}
Since $P(\xbs,t)$ is a linear combination of time-space harmonic polynomials $ Q_{\kbs}(\xbs, t),$ then  $P(\xbs,t)$ is in turn time-space harmonic.
Vice versa, suppose $P(\xbs,t)$ in \eqref{polynomial} a time-space harmonic polynomial with respect to $\{t \punt \mubs\}_{t \geq 0}.$ Then, there exist some coefficients $\{ c_{\kbs}\}$ such that 
$$
P(\xbs,t)  = \sum_{\kbs \leq \vbs} c_{\kbs} E[(\xbs - t \punt \mubs)^{\kbs}]  = \sum_{\kbs \leq \vbs} \left( \sum_{\kbs \leq \jbs \leq \vbs} 
\binom{\jbs}{\kbs} c_{\kbs} E[(-t \punt \mubs)^{\jbs - \kbs}] \right) \xbs^{\kbs}
$$
so that equality (\ref{prop5_multi}) follows.
\end{proof}
Theorem \ref{comb_multi} states a more general result: thanks to (\ref{(comb)}), any polynomial which is time-space harmonic with respect 
to the family $\{t \punt \mubs\}_{t \geq 0}$ can be expressed as a linear combination of the polynomials $Q_{\vbs}(\xbs,t).$
From (\ref{(comb)}), the coefficients of such a linear combination are
$$p_{\scriptscriptstyle\kbs}(0) = \sum_{\kbs \leq \jbs \leq \vbs}  \binom{\jbs}{\kbs} \, p_{\scriptscriptstyle\kbs}(0) \, E[\boldsymbol{\varepsilon}^{\jbs-\kbs}]$$
as $ 0 \punt \mubs \equiv \boldsymbol{\varepsilon}.$ Next section is devoted to some examples of multivariate time-space harmonic polynomials. 
In particular we generalize the class of  L\'evy-Sheffer systems and solve an open problem
addressed in \cite{BE} involving multivariate Bernoulli and Euler polynomials.  

%------------------------------------------------------------------------------------
\subsection{Examples}
%------------------------------------------------------------------------------------
%
%------------------------------------------------------------------------------------
\paragraph{Multivariate L\'evy-Sheffer systems.}
%------------------------------------------------------------------------------------
\begin{defn} \label{MTS}
A sequence of multivariate polynomials $\{V_{\kbs}(\xbs,t)\}_{t \geq 0}$ is a \emph{multivariate L\'evy-Sheffer system} if 
$$
1 + \sum_{k \geq 1} \sum_{\substack{\vbs \in \mathbb{N}_0^d \\ |\vbs| = k}} V_{\kbs}(\xbs,t)\frac{\zbs^{\kbs}}{\kbs!} =  [g(\zbs)]^t 
\exp\{ (x_1 + \cdots + x_d) [h(\zbs) - 1] \}
$$
where $g(\zbs)$ and $h(\zbs)$ are analytic in a neighborhood of $\zbs =\boldsymbol{0}$ and 
$$\left. \frac{\partial}{\partial z_i} h(\zbs) \right|_{\zbs = \boldsymbol{0}} \neq 0 \quad \hbox{for 
$i=1,2,\ldots,d$}.$$ 
\end{defn}
\begin{prop}
If $\mubs$ and $\nubs$ are $d$-tuples of umbral monomials such that $f(\mubs,\zbs)=g(\zbs)$ and $f(\nubs,\zbs)=h(\zbs)$ respectively, then 
\begin{equation}
V_{\kbs}(\xbs,t) = E[(t \punt \mubs +  (x_1 + \cdots + x_d) \punt \beta \punt \nubs)^{\kbs}].
\label{(LS)}
\end{equation}
\end{prop}
\begin{proof}
The result follows from Definition \ref{MTS}, since $f(t \punt \mubs, \zbs) = 
f(\mubs, \zbs)^t$ and the auxiliary umbra $(x_1 + \cdots + x_d)\punt \beta \punt \nubs$ has  
generating function \cite{dibruno} 
$f((x_1 + \cdots + x_d) \punt \beta \punt \nubs, \zbs) = \exp\{ (x_1 + \cdots + x_d) [f(\nubs, \zbs) - 1] \}.$
\end{proof}

The sequence $\{V_{\kbs}(\xbs,t)\}_{t \geq 0}$ is said to be a multivariate L\'evy-Sheffer system for the pair $\mubs$  and $\nubs.$ Equation (\ref{(LS)}) allows us to characterize the polynomial
$V_{\kbs}(\xbs,t)$ as a multivariate time-space harmonic polynomial with respect to a suitable symbolic L\'evy 
process. 

To this aim we need to introduce the \emph{multivariate compositional inverse} of a $d$-tuple $\nubs.$ 
Assume  $\chibs_{(i)}$ the $d$-tuple with all components equal to the augmentation umbra and only the $i$-th one equal
to the singleton umbra, that is $\chibs_{(i)} = (\varepsilon, \ldots, \chi, \ldots, \varepsilon).$
The multivariate compositional inverse of $\nubs$ is the umbral $d$-tuple $\deltabs=(\delta_1, \ldots, \delta_d)$ 
such that $\delta_i \punt \beta \punt \nubs \equiv \chibs_{(i)},$ for $i = 1, \ldots, d.$ By analogy
with the compositional inverse of an umbra, we denote the $d$-tuple $\deltabs$ with the symbol
$\nubs^{{\scriptscriptstyle <-1>}}$ and its $i$-th umbral monomial component with the symbol
$\nu_i^{{\scriptscriptstyle <-1>}}.$ 

\begin{thm}
The multivariate L\'evy-Sheffer polynomials  for the pair $\mubs$  and $\nubs$ are time-space harmonic  with respect to the symbolic multivariate L\'evy process 
$\{t \punt (\mu_1 \punt \beta \punt \nu_1^{{\scriptscriptstyle <-1>}} + \cdots + \mu_d \punt \beta \punt \nu_d^{{\scriptscriptstyle <-1>}})\}_{t \geq 0}.$
\end{thm}

\begin{proof}
As $f(\chibs_{(i)}, \zbs) = 1 + z_i$ for  $i = 1, \ldots, d,$ then  
$f(\mubs,\zbs)  = f(\mubs, (f(\chibs_{(1)}, \zbs) - 1, \ldots, f(\chibs_{(d)}, \zbs) - 1))$
so that $\mubs \equiv \mu_1 \punt \beta \punt \chibs_{(1)} + \cdots + \mu_d \punt \beta \punt 
\chibs_{(d)},$ see \cite{dibruno}. Thanks to the distributive property, from (\ref{(LS)}) 
\begin{eqnarray}
& & t \punt \boldsymbol\mubs +  (x_1 + \cdots + x_d) \punt \beta \punt \nubs \nonumber \\
& \equiv & t \punt (\mu_1 \punt \beta \punt \chibs_{(1)} + \cdots + \mu_d \punt \beta \punt \chibs_{(d)}) + (x_1 + \cdots + x_d) \punt \beta \punt \nubs \nonumber \\
& \equiv & t \punt (\mu_1 \punt \beta \punt \nu_1^{{\scriptscriptstyle <-1>}} \punt \beta \punt \nubs + \cdots + \mu_d \punt \beta \punt \nu_d^{{\scriptscriptstyle <-1>}} \punt \beta \punt \nubs)  + (x_1 + \cdots + x_d) \punt \beta \punt \nubs \nonumber  \\ 
& \equiv & [t \punt (\mu_1 \punt \beta \punt \nu_1^{{\scriptscriptstyle <-1>}} + \cdots + \mu_d \punt \beta \punt \nu_d^{{\scriptscriptstyle <-1>}}) + (x_1 + \cdots + x_d)] \punt \beta \punt \nubs.
\label{LS1}
\end{eqnarray}
Now in (\ref{(tshumbral_multi)}), set $-t \punt \mubs = \boldsymbol{\eta}.$ Then a multivariate time-space harmonic polynomial $Q_{\vbs}(\xbs,t)$ is such that
$Q_{\vbs}(\xbs,t) = E[ (x_1 + \eta_1)^{v_1} (x_2 + \eta_2)^{v_2} \cdots (x_d + \eta_d)^{v_d} ].$ So any multivariate time-space harmonic polynomial is linear combination of 
products $E[ (x_1 + \eta_1)^{v_1} (x_2 + \eta_2)^{v_2} \cdots (x_d + \eta_d)^{v_d}].$ The same results by expanding (\ref{LS1}) via (\ref{composition_multi}) and by observing
that
$$ V_k(\xbs,t) \simeq \{[(t \punt \mu_1 \punt \beta \punt \nu_1^{{\scriptscriptstyle <-1>}} + x_1)+ \cdots + (t \punt \mu_d \punt \beta \punt \nu_d^{{\scriptscriptstyle <-1>}}+ x_d )] \punt \beta \punt \nubs\}^{\kbs}.$$
\end{proof}

%------------------------------------------------------------------------------------
\paragraph{Multivariate Hermite polynomials.}
%----------------------------------------------------------------------------------
Let us consider the multivariate nonstandard Brownian motion $\{\X_t\}_{t \geq 0}$ such that $\X_t = C \, {\boldsymbol B}_t,$  where $C$ is a $d \times d$ matrix, whose determinant is not zero, and $\{{\boldsymbol B}_t\}_{t \geq 0}$ is the multivariate standard Brownian motion in ${\mathbb R}^d$ \cite{sato}. The moment generating function of $\{\X_t\}_{t \geq 0}$ is 
\begin{equation}
\varphi_t(\zbs) = \exp\left( 
\frac{t}{2}\zbs\Sigma\zbs^{T}\right), \quad \Sigma=C C^T.
\label{mgfmb}
\end{equation}
The multivariate nonstandard Brownian motion $\{\X_t\}_{t \geq 0}$ is a special multivariate L\'evy process. Indeed, as
shown in \cite{multilevy}, a multivariate L\'evy process admits a different symbolic representation 
with respect to $t \punt \mubs,$ if one prefers to highlight the role played by its multivariate cumulants. Indeed  any $d$-tuple $\mubs$ is such that $\mubs \equiv \beta \punt 
\boldsymbol{c}_{\mubs}$ with $\boldsymbol{c}_{\mubs}$ a $d$-tuple whose moments are cumulants
of the sequence $\{g_{\vbs}\}_{\vbs \in {\mathbb N}_0^d}.$ From (\ref{mgfmb}) we have
$$f(\boldsymbol{c}_{\mubs}, \zbs) = 1 + \frac{1}{2} \, \zbs\Sigma\zbs^{T}.$$
By considering the $d$-tuple $\deltabs$ such that $f(\deltabs,\zbs) = 1 +\zbs\zbs^{T}/2,$
then every multivariate nonstandard Brownian motion with covariance matrix $\Sigma$ is umbrally 
represented by the family of auxiliary umbrae $\{t \punt \beta \punt (\deltabs C^{T})\}_{t \geq 0}.$

There is a special family of multivariate polynomials which is strictly related to multivariate nonstandard Brownian motion: the generalized Hermite polynomials. The $\vbs$-th Hermite polynomial $H_{\vbs}(\boldsymbol{x}, \Sigma)$ is defined as
$$H_{\vbs}(\boldsymbol{x}, \Sigma) = (-1)^{|\vbs|} \frac{D_{\boldsymbol{x}}^{(\vbs)} \phi(\boldsymbol{x}; \boldsymbol{0}, \Sigma)}{{\phi(\boldsymbol{x}; \boldsymbol{0}, \Sigma)}},$$
where $\phi(\boldsymbol{x}; \boldsymbol{0}, \Sigma)$ denotes the multivariate Gaussian
density with $\boldsymbol{0}$ mean and covariance matrix $\Sigma$ of full rank $d.$

We consider the polynomials $\tilde{H}_{\vbs}(\boldsymbol{x}, \Sigma) = H_{\vbs}(\boldsymbol{x} \Sigma^{-1}, \Sigma^{-1})$
which are orthogonal with respect to $\phi(\boldsymbol{x}; \boldsymbol{0}, \Sigma),$ where $\Sigma^{-1}$ denotes the inverse  of $\Sigma.$

\begin{thm} \label{Her}
The family $\{\tilde{H}_{\vbs}^{(t)} (\xbs, \Sigma)\}_{t \geq 0}$ of generalized multivariate Hermite polynomials is time-space harmonic with 
respect to the symbolic nonstandard multivariate Brownian motion  $\{t \punt \beta \punt (\debs C^{\trasp})\}_{t \geq 0}.$
\end{thm}
\begin{proof}
The moment generating function of $\{\tilde{H}_{\vbs}^{(t)} (\xbs, \Sigma)\}_{t \geq 0}$ is \cite{dibruno}
$$
1 + \sum_{k \geq 1} \sum_{|\vbs| = k} \tilde{H}_{\vbs}^{(t)} (\xbs, \Sigma) \frac{\xbs^{\vbs}}{\vbs!} = \exp\left\{ \xbs \zbs^{\trasp} 
- \frac{t}{2} \zbs \Sigma \zbs^{\trasp} \right\}.
$$
The result follows since $\exp\left\{ \xbs \zbs^{\trasp}\right\} = f(\xbs, \zbs)$ and 
\begin{align*}
f(-t \punt \beta \punt (\debs C^{\trasp}), \zbs) = \exp\left\{\left[-\frac{t}{2}\zbs \Sigma \zbs^{\trasp}\right]\right\},
\end{align*}
so that $\tilde{H}_{\vbs}^{(t)}(\xbs, \Sigma) = E[(\xbs - t \punt \beta \punt (\debs C^{\trasp}))^{\vbs}].$
\end{proof} 
From Corollary \ref{zero} we also have $E[\tilde{H}^{(t)}_{\vbs}(t {\boldsymbol .} \beta \punt (\deltabs C^{T}), \, \Sigma)]=0.$
\begin{cor} 
$\tilde{H}_{\vbs}^{(t)}(\xbs, \Sigma) = E[(\xbs -1 \punt \beta \punt [\debs (t^{1/2} C^{\trasp})])^{\vbs}].$
\end{cor}
\begin{proof}
The result follows from Theorem \ref{Her}, since 
\begin{align*}
\exp\left\{- \frac{t}{2} \zbs \Sigma \zbs^{\trasp} \right\} & = \exp\left\{- \frac{t}{2} \zbs C C^{\trasp} \zbs^{\trasp} \right\}  = \left(\exp\left\{ f[\debs, \zbs(t^{1/2} C)] - 1 \right\}\right)^{\scriptscriptstyle{-1}} \\ 
& = \left(f(\beta \punt [\debs (t^{1/2} C^{\trasp})], \zbs) \right)^{\scriptscriptstyle{-1}} = f(-1 \punt \beta \punt [\debs (t^{1/2} C^{\trasp})], \zbs).
\end{align*}
\end{proof}
%
%------------------------------------------------------------------------------------
\paragraph{Multivariate Bernoulli polynomials.}
%----------------------------------------------------------------------------------
Multivariate Bernoulli polynomials are umbrally represented by umbral polynomials involving multivariate L\'evy processes 
\cite{BE}. Indeed, let $\iota$ be the Bernoulli umbra, that is the umbra whose moments are the Bernoulli numbers.
The \emph{multivariate Bernoulli umbra} $\bio$ is the $d$-tuple $(\iota, \ldots, \iota).$ 
For all $\vbs \in \mathbb{N}_0^d$ and $t \geq 0,$ the multivariate Bernoulli polynomial of 
order $\vbs$ is 
$$B_{\vbs}^{(t)}(\xbs)=E[(\xbs + t \punt \bio)^{\vbs}].$$ 
\begin{thm}
The family $\{B_{\vbs}^{(t)}(\xbs)\}_{t \geq 0}$ of multivariate Bernoulli polynomials is time-space harmonic with 
respect to the family  $\{- t \punt \bio\}_{t \geq 0}.$
\end{thm}
In particular, we have $E[\mathcal{B}_{\vbs}^{(t)}(- t \punt \bio)] = E\left\{ \mathcal{B}_{\vbs}^{(t)}[t \punt (-1 \punt \bio)] \right\} = 0.$
The auxiliary umbra $- t \punt \bio$ is the symbolic version of a L\'evy process. 
Indeed, $-1 \punt \bio$ is the umbral counterpart of a $d$-tuple identically distributed to 
$(U,\ldots, U),$ where $U$ is a uniform r.v. on the interval $(0,1).$ 
%
%------------------------------------------------------------------------------------
\paragraph{Multivariate Euler polynomials.}
%----------------------------------------------------------------------------------
Multivariate Euler polynomials are umbrally represented by umbral polynomials involving multivariate L\'evy processes 
\cite{BE}. Let $\eta$ be the Euler umbra, that is the umbra whose moments are the Euler numbers. Then 
the \emph{multivariate Euler umbra} $\boeta$ is the $d$-tuple $(\eta, \ldots, \eta).$ For all $\vbs \in \mathbb{N}_0^d$ and 
$t \in \Real,$ the multivariate Euler polynomial of order $\vbs$ is 
$$\mathcal{E}_{\vbs}^{(t)}(\xbs)=E\left\{ \left(\xbs + \frac{1}{2} [ t \punt (\boeta - \boldsymbol{u})] \right)^{\vbs}\right\}$$ 
with $\boldsymbol{u}=(u, \ldots, u)$ a vector of unity umbrae and $\boeta$ the multivariate Euler umbra.
\begin{thm}
The family $\{\mathcal{E}_{\vbs}^{(t)}(\xbs)\}_{t \geq 0}$ of multivariate Euler polynomials is time-space harmonic with 
respect to the family  $\{\frac{1}{2} [t \punt (\boldsymbol{u} - 1 \punt \boeta)]\}_{t \geq 0}.$
\end{thm}
The auxiliary umbra $\frac{1}{2} [t \punt (\boldsymbol{u} - 1 \punt \boeta)]$ is the symbolic version of a L\'evy process. 
Indeed, $\frac{1}{2} (\boldsymbol{u} - 1 \punt \boeta)$ is the umbral counterpart of a $d$-tuple identically distributed to 
$(Y,\ldots, Y),$ where $Y$ is a Bernoulli r.v. of parameter $1/2.$  
%------------------------------------------------------------------------------------------------------------------
%bibliography
%------------------------------------------------------------------------------------------------------------------

\end{document}